\documentclass[11pt]{amsart}
\usepackage{latexsym,amssymb,amsmath}
\usepackage{amsmath,amsfonts}
\usepackage{verbatim}
\usepackage{enumitem}
\usepackage[margin=1.75in]{geometry}

\newtheorem{theorem}{Theorem}[section]
\newtheorem{lemma}[theorem]{Lemma}
\newtheorem{proposition}[theorem]{Proposition}

\newtheorem{definition}[theorem]{Definition}
\newtheorem{remark}[theorem]{Remark}

\allowdisplaybreaks

\usepackage{hyperref}
\hypersetup{
  colorlinks   = true,    % Colors links instead of boxes
  linkcolor    = blue,    % Color for internal links
  urlcolor     = blue,
  citecolor    = red     }% Color for citations 

\newcommand{\N}{\mathbb{N}}
\newcommand{\R}{\mathbb{R}}

\newcommand{\I}{\operatorname{I}}

\newcommand{\lb}{\langle}
\newcommand{\rb}{\rangle}
\renewcommand{\d}{\operatorname{d}\!}

\renewcommand{\Re}{\operatorname{Re}}

\renewcommand{\hat}{\,\widehat}
\newcommand{\wt}{\widetilde}

\begin{document}

\title[3--gKdV on $\R^+$]{Sharp well--posedness for the generalized KdV of order three on the half line}

\author{E. Compaan}
\address{Department of Mathematics, Massachusetts Institute of Technology} 
\email{compaan@mit.edu}

\author{N. Tzirakis}
\address{Department of Mathematics, University of Illinois at Urbana-Champaign}
\email{tzirakis@illinois.edu}



\date{\today}

\begin{abstract} In this paper we study the generalized Korteweg de Vries (KdV) equation with the nonlinear term of order three: $(u^{3+1})_x$. We prove sharp local well--posedness for the initial and boundary value problem posed on the right half line. We thus close the gap in the well--posedness theory of the generalized KdV which remained open after the seminal work of Colliander and Kenig in \cite{CK}. 
\end{abstract}

\maketitle
%\tableofcontents

\section{Introduction}

The generalized Korteweg--de Vries equation (gKdV) is the real-valued model
\begin{equation} \label{eq:gkdv}
\begin{cases}
u_t + u_{xxx} + (u^{k+1})_x = 0, \quad (x,t) \in \R^+ \times \R^+, \\
u(x, 0) = u_0(x) \in H^{s}(\R^+),\\
u(0,t) = g(t) \in H^{\frac{s+1}{3}}(\R^+).
\end{cases}
\end{equation}
The initial--boundary value problem \eqref{eq:gkdv} on the half line with a first order transport term, physically
models the evolution of small amplitude, shallow, long water waves propagating
in a channel with forcing applied at the left end \cite{hs}. The well--posedness theory was developed initially by Bona and Winther in \cite{bw, bw1}, and much later by Colliander and Kenig in \cite{CK}. The work of Colliander and Kenig introduced a new method to solve initial--boundary value problems (IBVP) for nonlinear
dispersive partial differential equations by recasting these problems as initial
value problems with appropriate forcing terms. Their work showed that it is possible to obtain well--posedness results for the right half line that match the results that have been obtained on the full line \cite{KPV}, \cite{kpv1}, \cite{G}. Their paper left open the cases when $k=1$ and $k=3$. The $k=1$ case was later completed by Holmer in \cite{jh}. The importance of the work in \cite{CK} is that the authors wrote down a representation solution--formula with certain terms enforcing the boundary conditions. They then employed robust dispersive techniques which were introduced by Bourgain in \cite{bou1} to obtain solutions of the KdV equation on the real line and the circle with nonsmooth data. 

Another method to solve IBVP for disperesive PDE was employed by Bona, Sun and Zhang in \cite{bsz}. In this method, after extending the initial data in the whole real line, one employs the Laplace and Fourier transform method to solve the linear problem in the usual Sobolev spaces. Then an equivalent integral representation (Duhamel's formula) of the solution is written using the trace of the linear boundary operator and the nonlinearity. Trace theorems dictate the regularity of the boundary function and the nonlinearity of each particular PDE dictates the Banach space in which a fixed point of the integral equation is sought. The existence of a fixed point can be established if one has appropriate linear and nonlinear estimates in a suitable function space. For dispersive equations with derivatives, the $X^{s,b}$ spaces of Bourgain, \cite{bou1} are the right choice in most cases. These spaces can be employed using ideas from \cite{CK}.  Using this approach, Bona et. al recently solved the KdV equation and the nonlinear Schr\"odinger equations on the half line, along with many other interesting problems, see \cite{bsz} and the references therein. It is this approach that we follow on our paper. 

As we have already mentioned, the other case that was left open in \cite{CK} was the case $k=3$. On the full line, the sharp local well--posedness for $k=3$ was completed by Gr\"unrock in \cite{G}.
 In particular Gr\"unrock showed well--posedness on the full line for $s > -\frac16$. This result is sharp up to the endpoint due to scaling considerations, \cite{bkpsv}. For the half line the authors in \cite{CK} obtain well--posedness for $H^{\frac1{12}}$ initial data and thus a gap remained with the known $\R$ theory. In this paper we are closing this gap for the problem on the right half-line. 

More precisely we are interested in the following equation, known as the 3--gKdV, on the right half--line:
\begin{equation} \label{eq:3gkdv}
\begin{cases}
u_t + u_{xxx} + (u^4)_x = 0, \quad (x,t) \in \R^+ \times \R^+, \\
u(x, 0) = u_0(x) \in H^{s}(\R^+), \\
u(0,t) = g(t) \in H^\frac{s+1}{3}(\R^+). 
\end{cases}
\end{equation}
The data $(u_0,g)$ will be taken in the space $H^s_x(\R^+)\times H^{\frac{s + 1}{3}}_t(\R^+)$ with the additional compatibility conditions $h(0)=g(0)$ when $\frac{1}{2}< s<\frac52$. These compatibility conditions are necessary since the solutions we are interested in are continuous space-time functions for $s>\frac12$. To state the main theorem of this paper we start with a definition.

\begin{definition}
We say that the 3--gKdV  equation \eqref{eq:3gkdv} is locally well--posed in $H^{s}(\R^+)$ with $s < \frac52$, if for any $(u_0,g) \in H^s_x(\R^+) \times H^{\frac{s + 1}{3}}_t(\R^+)$, with the additional compatibility conditions mentioned above, the equation $\Phi( u) = u$, where $\Phi$ is defined by \eqref{eq:duhamelForm}, has a unique solution in 
\[ C_t^0H^s_x \cap C_x^0H_t^{\frac{s+1}{3}}\cap  X^{s,\frac12}_T ,\]
for some sufficiently small $T$, dependent only on the norms of the initial and boundary data. Furthermore, the solution depends continuously on the initial and boundary data. In addition, the solution on $\R^+$ is independent of the choice of extension used to define \eqref{eq:duhamelForm}. 
\end{definition}

The main result of our paper is the following.

\begin{theorem} \label{main_thrm}
 For any $s \in \left(-\frac16, 2 \right)$, $s \neq \frac12, \frac32$, the equation \eqref{eq:3gkdv} is locally well--posed in $H^s(\R^+)$. Moreover, we have the following smoothing estimate. For any $a < \frac{1}{4}+\min(0,\frac{3s}{2})$, any $u_0 \in H_x^{s}(\R^+)$, and any $g \in H_t^{\frac{s+1}{3}}(\R^+)$ we have that
 \begin{equation*}
  u - W_0^t(u_0,g) \in C_0^tH^{s + a}_x,
 \end{equation*} 
 where the linear IBVP solution $W_0^t(u_0,g)$ is given by equation \eqref{eq:fullairy}. %In addition, the solutions are independent of the extensions of the initial data.
\end{theorem}
 
\begin{remark}
The method of the proof of the Theorem is quite general. Using similar arguments we can study the regularity properties of nonlinear dispersive partial differential equations (PDE) on a half line using the tools that are available in the case of the real line, where the PDE are fully dispersive.
\end{remark}
 
\begin{remark}
We should note that the smoothing estimate is not just a byproduct of our multilinear $L^2$ convolution estimates but it is also instrumental in proving the sharp well--posedness theory. In addition, the nonlinear smoothing is used in order to prove that the solutions of \eqref{eq:3gkdv} are unique.
\end{remark} 

\begin{remark}
The proof of the Theorem actually proves that the solution lies in 
\[    
\Bigl[ C^0_tH^s_x \cap C^0_x H^\frac{s+1}3_t \cap X^{s, \frac12} \Bigr]+ \Bigl[ C^0_tH^{s+a}_x \cap C^0_x H^\frac{s+a+1}3_t \cap X^{s+a, \frac12+} \Bigr]. \]
which is a subspace of 
\[ C_t^0H^s_x \cap C_x^0H_t^{\frac{s+1}{3}}\cap  X^{s,\frac12} .\]
For the definition of the $X^{s,b}$ space see the next section.
\end{remark}

\begin{remark}
As expected the smoothing disappears at the upper endpoint $s=-\frac16$ The reader can consult \cite{ETbook} for  many examples of dispersive PDE that enjoy nonlinear smoothing properties at regularities equal to the regularities of the sharp local well--posedness theory. 
\end{remark}

To prove the above theorems we rely on a Duhamel formulation of the nonlinear system adapted to the boundary conditions. This expresses the nonlinear solution as the superposition of the linear evolutions which incorporate the boundary and the initial data with the nonlinearity. Thus, we first solve two linear problems  by a combination of Fourier and Laplace transforms, \cite{ET}, after extending the initial data to the whole line. The idea is then to use the restricted norm method in the Duhamel formula. The novelty in our approach is the following. In the general theory of the KdV equation the dispersive weight of the $X^{s,b}$ norm (in our case $\lb \tau+\xi^3\rb^{b}$) cannot be arbitrary. In particular for initial and boundary value problems of KdV type, the nonlinear Duhamel term  is never in $X^{s,b}$ for $b>\frac12$, see \cite{ET}. But to obtain the nonlinear estimates on the real line in the generalized KdV theory one needs to work with $b>\frac12$. To bypass this issue we are splitting our solution into a linear and a nonlinear part. The linear part can be put in an $X^{s_0,b}$ space with $b>\frac12$ but with $s_0<s$. To close the argument we should be able to recover the loss of derivative in the nonlinear interactions. But this can be done by our smoothing estimates. The fact that we can take $b>\frac12$ in our multilinear estimates is the heart of the matter, since then we can take advantage of the dispersive estimates on the full line and complete the proof. The details of this estimation is presented in Proposition \ref{nonlin}.

In addition the smoothing estimates can also clarify the uniquness of the IBVP. To understand this problem we note that the uniqueness of the solutions thus constructed is not immediate since we do not know that the fixed points of the Duhamel operators have restrictions on the half line which are independent of the extension of the data. For the case of more regular data the uniqueness property of the solution can be proved by standard energy arguments. For less regular data we take advantage of the smoothing estimate we establish in Theorem \ref{main_thrm} to obtain uniqueness all the way down to the local theory threshold. We remark that this iteration is successful because the full nonlinear estimate we provide remains valid for any $s>-\frac{1}{6}$, matching thus the regularity of the local theory.

We now discuss briefly the organization of the paper.  In Section \ref{notation}, we introduce some notation and the function spaces that we use to obtain the well--posedness of the IBVP. In Section \ref{solSetup} we define the notion of the solution. More precisely we set up the integral representation (Duhamel's formula) of the nonlinear solution map that we later prove is a contraction in an appropriate metric space. We obtain the solution as a superposition of a linear and a nonlinear evolution. The solution of the linear IBVP can be found by a direct application of the Fourier and the Laplace transform methods. Section \ref{estimates} presents the linear and nonlinear a priori estimates that we use to iterate the solution using the restricted norm method appropriately modified for our needs. In Section \ref{locTheory} we prove the local well--posedness property of the solutions by splitting the flow into a linear and a nonlinear evolution. This completes the existence part of Theorem \ref{main_thrm}. Uniqueness is proved in Section \ref{uniq}. Section \ref{proofs} is the most technical part of the paper. There we prove the main nonlinear estimate, Proposition \ref{nonlin} and we also establish the proof of Proposition \ref{duhamelEst} which is used in the iteration process.

\section{Notation \& Function Spaces}\label{notation}
The one--dimensional Fourier transform is defined by 
\[ \widehat{f}(\xi) = \mathcal{F}_x f(\xi) = \int_{\R} e^{-i x \xi} f(x) \d x. \]
We set 
\[\lb \xi \rb = \sqrt{ 1 + |\xi|^2} \approx 1 + |\xi|,\]
and define the non--homogeneous and homogenous Sobolev space $H^s$ norms respectively
by 
\[ \| f\|_{H^s(\R)} = \left( \int \lb \xi \rb^{2s} |\hat{f}(\xi)|^2 \d \xi \right)^{\frac12}\] 
and
\[ \| f\|_{\dot{H}^s(\R)} = \left( \int |\xi|^{2s} |\hat{f}(\xi)|^2 \d \xi \right)^{\frac12}.\] 
The characteristic function on the positive half-line $[0,\infty)$ is denoted by $\chi$, and Sobolev spaces $H^s(\R^+)$ on the half line for $s > -\frac12$ are defined as follows:
\begin{align*}
H^s(\R^+) &= \Bigl\{ g \in \mathcal{D}(\R^+) \; : \; \text{there exists } \tilde{g} \in H^s(\R) \text{ with } \tilde{g} \chi = g \Bigr\},  \\
\|g\|_{H^s(\R^+)}&= \inf \Bigl\{ \|\tilde{g} \|_{H^s(\R)} \; : \; \chi \tilde{g} = g \Bigr\}. 
\end{align*}
We will also use the Fourier restriction norm spaces (\cite{bou1}) corresponding to the Airy flow. These are defined for functions on $\R_x \times \R_t$ by the norm 
\[ \|u\|_{X^{s,b}} = \| \lb \xi \rb^s \lb \tau - \xi^3 \rb^b \hat{u}(\xi,\tau) \|_{L^2_{\xi,\tau}}.\] 

For our nonlinear estimate, we make use of the operators $D$, $I^\frac12$, and $I^\frac12_-$, which are defined by the following Fourier transform formulae:
\begin{align}\label{eq:operatornames}
\begin{split}
\hat{D^sf}(\xi) &= \lb \xi \rb^s \hat{f}(\xi), \\
\hat{I^\frac12(f)}(\xi) &= |\xi|^\frac12 \hat{f}(\xi),\\
\hat{I_-^\frac12(f,g)}(\xi) &= \int |2\xi_1 - \xi|^\frac12 \hat{f}(\xi_1) \hat{g}(\xi-\xi_1) \d \xi_1.
\end{split}
\end{align}
The operator $I_-^\frac12$ was introduced by Grunrock in \cite{G}. 

Let $\rho \in C^\infty$ be a cut-off function such that $\rho = 1$ on $[0, \infty)$ and $\operatorname{supp} \rho \subset [-1, \infty)$. Let $\eta \in C^\infty$ be a bump function such that $\eta = 1$ on $[-1,1]$ and $\operatorname{supp} \eta \subset [-2,2]$. Also, define $\eta_T = \eta(\cdot/T)$.

Finally, the notation $a \lesssim b$ indicates that $a \leq Cb$ for some absolute constant $C$. The expression $a \gtrsim b$ is defined similarly, and $a \approx b$ means that $a \lesssim b$ and $a \gtrsim b$. The notation $a+$ indicates $a + \epsilon$, where $\epsilon$ can be arbitrarily small. We define $a-$ similarly.

\section{Solution Formulations} \label{solSetup}
Using the Fourier inversion formula it is a standard fact that for smooth and decaying initial data the solution to the initial value problem,
\begin{equation}\label{eq:lineAiry}
\begin{cases}
u_t + u_{xxx}  = 0, \quad (x,t) \in \R \times \R, \\
u(x, 0) = u_{0}(x),
\end{cases}
\end{equation} 
is given by 

$$
W_\R^t u_{0}(x)=e^{-t\partial_{xxx}} u_{0}(x)= \mathcal F^{-1}\big[e^{it\xi^{3}} \widehat u_{0}(\cdot)\big](x).
$$
On the other hand, the following formula for the solution to the corresponding linear IBVP problem with zero initial data is known.

\begin{lemma}
The solution to the linear equation
\begin{equation}\label{eq:linAiry}
\begin{cases}
v_t + v_{xxx}  = 0, \quad (x,t) \in \R^+ \times \R^+, \\
v(x, 0) = 0 , \\
v(0,t) = g(t) \in H^\frac{s+1}{3}(\R^+),
\end{cases}
\end{equation}
can be written in the form
\begin{equation} \label{eq:IBVPsol} W_0^t(0,g) = \frac3\pi \Re \int_0^\infty e^{i \mu^3 t} e^{-i\mu x \cos(\pi/3)} e^{ - \mu x \sin(\pi/3)} \rho(\mu x ) \;  \hat{\chi g}(\mu^3)  \mu^2 \d \mu.\end{equation}
\end{lemma}

\begin{proof}
We will first show that the solution can be written in the form
\begin{equation}\label{eq:IBVPsol0}
v(x,t)= \frac3\pi \Re \int_0^\infty e^{i \mu^3 t} e^{-i\mu x \cos(\pi/3)}e^{- \mu x \sin(\pi/3)} \hat{\chi g}(\mu^3) \mu^2 \d \mu. 
\end{equation}
For a derivation of \eqref{eq:IBVPsol0} using the unified transform method, see \cite{FHM}. For completeness, we also provide a derivation using the Laplace transform below. 

Before proceeding to this calculation, we note that the above integral does not necessarily converge for all $x <0$. To obtain a integral which converges for all $x \in \R$, we multiply the integrand by the cut-off function $\rho$ and note that for $x \geq 0$, we have 
\begin{equation*}v(x,t) = \frac3\pi \Re \int_0^\infty e^{i \mu^3 t} e^{-i\mu x \cos(\pi/3)} e^{ - \mu x \sin(\pi/3)} \rho(\mu x ) \;  \hat{\chi g}(\mu^3)  \mu^2 \d \mu.\end{equation*} 
Thus we obtain the formula given in the Lemma. We will use this version of the solution formula in our work.

To establish the formula \eqref{eq:IBVPsol0} for the solution of the linear problem \eqref{eq:linAiry}, we use the Laplace transform, which is defined for functions on $[0,\infty)$ by 
\[ \wt{v}(\lambda) = \int_0^\infty e^{-\lambda t} v(t) \d t, \qquad \operatorname{Re}(\lambda) \geq 0. \]

Taking the Laplace transform of \eqref{eq:linAiry}, we obtain
\begin{equation*}
\begin{cases}
\lambda \wt{v}(x,\lambda) + \wt{v}_{xxx}(x,\lambda) = 0 \\
\wt{v}(0,\lambda) = \wt{g}(\lambda). 
\end{cases}
\end{equation*}
This ordinary differential equation has characteristic function $\lambda + w^3$. Since we are concerned with solutions which decay at infinity, the only relevant root is $w = - \lambda^{\frac13}$, where the third root is defined with a branch cut along the negative real axis. Thus we have 
\[ \wt{v}(x,\lambda) = e^{-\lambda^{\frac13}x}\; \wt{g}(\lambda).\]
By Mellin inversion,
for any $\sigma >0$ we have
\[ v(x,t) = \frac1{2 \pi i} \int_{\sigma - i \infty} ^{\sigma + i \infty} e^{\lambda t - \lambda^\frac13 x}\;  \wt{g} (\lambda) \d \lambda.\]
Parametrizing the contour by $\lambda(\mu) = \sigma + i\mu^3$ for $\mu \in (-\infty, \infty)$ and letting $\sigma \to 0$, we arrive at
\begin{align*} v(x,t) &=  \frac3{2 \pi } \left( \int_{-\infty} ^{0} \! \! e^{i\mu^3 t - \mu e^{-i\pi/6} x}\;  \wt{g} (i\mu^3) \mu^2\d \mu   +   \int_{0} ^{ \infty} \!\! e^{i\mu^3 t - \mu e^{i\pi/6} x}\;  \wt{g} (i\mu^3) \mu^2\d \mu \right)
\end{align*}
Changing variables in the first integral and simplifying, we obtain
\begin{align*}
v(x,t) &=  \frac3{2 \pi } \int_{0} ^{\infty} \Bigl[ e^{-i\mu^3 t + \mu e^{-i\pi/6} x} \;  \wt{g} (-i\mu^3)+  e^{i\mu^3 t - \mu e^{i\pi/6} x}\;  \wt{g} (i\mu^3) \Bigr]   \mu^2\d \mu \\ 
&=  \frac3{ \pi } \operatorname{Re} \int_{0} ^{\infty}   e^{i\mu^3 t - \cos(\pi/6) \mu x  - i\sin(\pi/6) \mu x} \;   \hat{\chi g} (\mu^3)   \mu^2\d \mu. 
\end{align*}
Since $\cos(\pi/3) = \sin(\pi/6)$ and $\cos(\pi/6) = \sin(\pi/3)$, we have obtained \eqref{eq:IBVPsol0}. 
\end{proof}

We now construct the unique solution of the linear initial-boundary value problem
\begin{equation}\label{eq:fullairy}
\begin{cases}
u_t + u_{xxx}  = 0, \quad (x,t) \in \R^+ \times \R^+, \\
u(x, 0) = u_{0}(x)  \in H^s(\R^+),\\
u(0,t) = g(t) \in H^\frac{s+1}{3}(\R^+),
\end{cases}
\end{equation}
which we denote by    $W_{0}^t(u_{0},g)$,  for $t\in [0,1]$. Note that
\begin{equation} \label{fulllinear}
W_0^t(u_{0},g)=W_0^t(0,g-p)+W_\R^t u_{0,e},
\end{equation}
where $u_{0,e}$ is an $H^s$ extension of $u_{0}$ to the full line satisfying $$\|u_{0,e}\|_{H^s(\R)}\lesssim \|u_0\|_{H^s(\R^+)}$$ and $p$ is defined by $p(t)= \eta(t) \bigl[W_\R^t u_{0,e}\bigr]\Big|_{x=0}$. Note that $p$ is well-defined and is in  $H^{\frac{s+1}{3}}(\R^+)$ by Lemma~\ref{kato} below.

For the full nonlinear problem \eqref{eq:3gkdv}, we have the Duhamel formula 
\begin{multline}\label{eq:duhamelForm0}
u(x,t) = \eta_T(t) W_\mathbb{R}^t u_{0,e} + \eta_T(t) W_0^t(0, g - p - q) \\+ \eta_T(t) \int_0^t W_\R^{t-t'} \bigl[ \partial_x(u^4) \bigr] \d t'  ,
\end{multline}
where 
\begin{align*}
p(t) = \eta(t) \bigl[ W_\R^t u_{0,e}(x) \bigr]\Big|_{x=0} \qquad
q(t) = \eta(t) \left[ \int_0^t W_\R^{t-t'} \bigl[ \partial_x(u^4) \bigr] \d t'\right] \Bigg|_{x=0}.
\end{align*}

As mentioned in the introduction, it will be necessary to break the problem down into linear and nonlinear parts, so we will not be using this formulation directly. However, we will base our formulae on this result. For more details, see the discussion in Section \ref{locTheory}.

\section{{A Priori} Estimates} \label{estimates}

In order to complete our contraction argument, we will require a number of estimates for the linear and nonlinear terms which comprise the Duhamel formula \eqref{eq:duhamelForm0}. In the following, we first collect all relevant linear estimates with any necessary proofs. We then proceed to the more complex nonlinear estimates, which will be proved later.

\subsection{Linear Estimates}

We begin with the following standard $X^{s,b}$ space estimates on $\R$.
\begin{lemma}[{\cite[Lemma 2.8]{tao}}] \label{RlinXsb}
 For any $s$ and any $b$, we have 
 \[ \| \eta(t) W_\R^t u_0 \|_{X^{s,b}} \lesssim \| u_0 \|_{H^s(\R)}. \]
\end{lemma}

\begin{lemma}[{\cite[Lemma 2.12]{tao}}] \label{RduhamelXsb}
For any $s\in \R$ and $b > \frac12$, we have
\begin{equation*}
\Big\| \eta(t) \int_0^t  W_\R^{t-t'} F(t^\prime ) dt^\prime \Big\|_{X^{s,b} }\lesssim   \|F\|_{X^{s,b-1} }.
\end{equation*}
\end{lemma}

\begin{lemma}[{\cite[Lemma 2.11]{tao}}]\label{powerT}
For any $T\leq1$, and $-\frac12<b_1<b_2<\frac12$, we have
\begin{equation*}
\|\eta_T(t) F \|_{X^{s,b_1}}\lesssim T^{b_2-b_1} \|F\|_{X^{s,b_2}}.
\end{equation*}
\end{lemma}

We also have a Kato smoothing estimate which describes the temporal regularity of the Airy flow. 
\begin{lemma}[{\cite[Lemma 4.1]{CK}}] \label{kato}
 For any $s \geq -1$, we have
 \[ \| \eta(t) W_\R^t u_0 \|_{L^\infty_x H^{\frac{s+1}{3}}_t} \lesssim \| u_0 \|_{H^s(\R)}. \]
\end{lemma}

To bound the solution to the linear initial--boundary--value problem in $X^{s,b}$ spaces, we have the following result.
\begin{lemma} \label{linXsb}
 For any $b \in \left(\frac16,\frac56\right)$ and any $s \geq \max\left\{-1, \frac12-3b\right\}$, we have 
 \[ \| \eta(t) W^t_0(0,g) \|_{X^{s,b}} \lesssim \| \chi g \|_{{H}^{\frac{2s+6b-1}{6} }(\R)}. \] 
\end{lemma}

\begin{proof}
Recall the formula \eqref{eq:IBVPsol} for $W^t_0$. Let $f(y) =  e^{-i y \cos(\pi/3)} e^{ - y \sin(\pi/3)} \rho(y)$. Note that $f$ is a Schwarz function. Then we have
 \[ \eta(t) W_0^t(0,g(x)) = \eta(t) \frac3\pi \Re \int_0^\infty e^{i \mu^3 t} f(\mu x) \;  \hat{\chi g}(\mu^3)  \mu^2 \d \mu \]
and 
\[ \hat{\eta W_0^t}(0,g)(\xi, \tau) = \frac3\pi \Re \int_0^\infty \hat{\eta}(\tau - \mu^3)  \hat{f}(\xi / \mu) \;  \hat{\chi g}(\mu^3)  \mu \d \mu .\] 

\noindent{\underline{\textbf{Case 1: $s=0$ with $|\xi| \lesssim 1$ and $|\mu| \lesssim 1$:}}}
In the this case, note that 
\[ |\hat{\eta}(\tau - \mu^3)| \lesssim \lb \tau - \mu^3 \rb^{-3} \approx \lb \tau \rb^{-3}.\]
Thus on the region where $|\xi|, |\mu| \lesssim 1$, we have the bound 
\begin{align*}
 \Big\| \lb \tau - \xi^3 \rb^b \hat{\eta W_0^t}(0,g)(\xi, \tau) \Big\|_{L^2_{|\xi| \lesssim1,\tau}} 
 &\lesssim \left\| \lb \tau \rb^{b-3} \int_{0}^1 \frac{\mu^4 | \hat{\chi g}(\mu^3) |}{|\mu|^3 + |\xi|^3} \d \mu \right\|_{L^2_{|\xi|\lesssim 1, \tau}} \\
 &\lesssim \left\|  \int_{0}^1 \frac{\mu^4}{|\mu|^3 + |\xi|^3} | \hat{\chi g}(\mu^3) \d \mu \right\|_{L^2_{|\xi|\lesssim 1}}.
\end{align*}
Using Minkowski's inequality to take the $L^2_\xi$ norm inside and noting that
\[ \int_{-1}^1 \frac{\d \xi}{(|\mu|^3 + |\xi|^3)^2} \lesssim \frac1{|\mu|^5}, \] 
we arrive at the bound
\begin{equation} \label{eq:smallfreqcase}
\left|  \int_{0}^1\mu^\frac32 | \hat{\chi g}(\mu^3) \d \mu \right| \approx \left| \int_{0}^1 |z|^{-\frac16} | \hat{\chi g} (z) | \d z \right| 
 \lesssim \| \chi g \|_{L^2}.
\end{equation}
Since $b > \frac16$, this can be bounded by $ \| \chi g \|_{H^{b - \frac16}}$ as desired.

\noindent{\underline{\textbf{Case 2: $s=0$ with $|\xi| \gtrsim 1$ or $|\mu| \gtrsim 1$:}}} In this case, note that since $f$ and $\hat{\eta}$ are Schwarz functions, we have 
\[ |\hat{f}(\xi/\mu)| \lesssim \frac{|\mu|^3}{|\mu|^3 + |\xi|^3} \]
and 
\[ |\hat{\eta}(\tau - \mu^3) | \lesssim \lb \tau - \mu^3 \rb^{-2} \lesssim \lb \tau - \mu^3 \rb^{-2 + b} \lb \tau - \xi^3 \rb ^{-b} \lb \mu^3 - \xi^3 \rb^b. \] 
In light of the these bounds, the problem reduces estimating
\begin{align*}
 \Big\| \lb \tau - \xi^3 \rb^b \hat{\eta W_0^t}&(0,g)(\xi, \tau) \Big\|_{L^2_{\xi,\tau}} \\
 &\lesssim  
 \left\| \lb \tau - \xi^3 \rb^b \int_0^\infty \hat{\eta}(\tau - \mu^3)  \hat{f}(\xi / \mu) \;  \hat{\chi g}(\mu^3)  \mu \d \mu \right\|_{L^2_{\xi,\tau}} \\
 &\lesssim
  \left\| \int_0^\infty  \lb \tau - \mu^3 \rb^{-2 + b} \lb \mu^3 - \xi^3 \rb^b  \frac{|\mu|^4}{|\mu|^3 + |\xi|^3} \; \left| \hat{\chi g}(\mu^3)  \right| \d \mu \right\|_{L^2_{\xi,\tau}} 
  \end{align*}
Since we are in the case where $|\xi| \gtrsim 1$ or $|\mu| \gtrsim 1$, we see that the above quantity is bounded by
\begin{align} \label{eq:higherfreqcase}
  \left\| \int_0^\infty  \lb \tau - \mu^3 \rb^{-2 + b}  \frac{|\mu|^4}{\left( |\mu|^3 + |\xi|^3 \right)^{1-b}} \; \left| \hat{\chi g}(\mu^3)  \right| \d \mu \right\|_{L^2_{\xi,\tau}}.
\end{align}
 Moving the $L^2_\xi$ norm inside the integral and then using Young's inequality, this is bounded by 
\begin{multline*}
 \left\| \int_0^\infty  \lb \tau - \mu^3 \rb^{-2 + b}  |\mu|^{\frac32 + 3b} \; \left| \hat{\chi g}(\mu^3)  \right| \d \mu \right\|_{L^2_\tau} 
 \\ \approx 
 \left\| \int_0^\infty  \lb \tau - z \rb^{-2 + b}  |z|^{-\frac16 + b} \; \left| \hat{\chi g}(z)  \right| \d z \right\|_{L^2_\tau} 
 \lesssim \| \chi g \|_{\dot{H}^{b - \frac16}(\R)} \lesssim \| \chi g \|_{H^{b - \frac16}(\R)},
 \end{multline*}
 where the last inequality holds because $b > \frac16$. This establishes the claim in the $s=0$ case. 
 
\noindent{\underline{\textbf{Case 3: $s>0$:}}} For general positive $s$, note that if $s \in \N$, we have 
 \[ \left| \partial^s_x \Bigl[\eta(t) W_0^t(0,g(x)) \Bigr] \right| \lesssim \left| \eta(t) \int_0^\infty e^{i \mu^3 t} f^{(s)}(\mu x) \;  \hat{\partial^{s/3} \chi g}(\mu^3)  \mu^2 \d \mu \right|. \]
Thus the argument used for the $s=0$ case applies when $s \in \N$. To obtain the desired conclusion for any positive $s$, we interpolate.
 
\noindent{\underline{\textbf{Case 4: $s<0$ with $|\xi| \lesssim 1$ and $|\mu| \lesssim 1$:}}} In this case, the $\lb \xi \rb^s$ multiplier is $\approx 1$, so we may argue just as above to arrive at \eqref{eq:smallfreqcase}. Since $\frac{2s+6b-1}{3} \geq 0$, the result of \eqref{eq:smallfreqcase} is sufficient.  

\noindent{\underline{\textbf{Case 5: $s<0$ with $|\xi| \gtrsim 1$ or $|\mu| \gtrsim 1$:}}} In this case, we argue as in Case 2 to arrive at a bound analogous to \eqref{eq:higherfreqcase}, i.e.
\begin{align} \label{eq:negCase}
  \left\| \int_0^\infty  \lb \tau - \mu^3 \rb^{-2 + b}  \frac{\lb \xi \rb^s |\mu|^4}{\left( |\mu|^3 + |\xi|^3 \right)^{1-b}} \; \left| \hat{\chi g}(\mu^3)  \right| \d \mu \right\|_{L^2_{\xi,\tau}}.
\end{align}
Note that 
\begin{align*}
\left\| \frac{\lb \xi \rb^s |\mu|^4}{\left( |\mu|^3 + |\xi|^3 \right)^{1-b}} \right\|_{L^2_\xi}^2 &= |\mu|^8 \int  \frac{\lb \xi \rb^{2s}}{\left( |\mu|^3 + |\xi|^3 \right)^{2(1-b)}} \d \xi \\
&\approx |\mu|^{3 + 6b} \int  \frac{\lb \mu w \rb^{2s}}{\lb w \rb^{6(1-b)}} \d w 
\end{align*}
If $|\mu| \gtrsim1$, then we have $\lb \mu w \rb^{2s} \lesssim |\mu|^{2s} \lb w \rb^{2s}$, so \eqref{eq:negCase} can be bounded by
\[ \left\| \int_0^\infty  \lb \tau - \mu^3 \rb^{-2 + b}  |\mu|^{\frac32 + 3b +s} \; \left| \hat{\chi g}(\mu^3)  \right| \d \mu \right\|_{L^2_{\tau}}\] 
and the argument can be completed as in Case 2. 

For the region where $|\mu| \leq 1$, the quantity in \eqref{eq:negCase} can be bounded by 
\begin{align*} 
  \left\| \int_0^1  \lb \tau  \rb^{-2 + b}  \frac{ |\mu|^4}{\lb \xi \rb^{3-3b-s}} \; \left| \hat{\chi g}(\mu^3)  \right| \d \mu \right\|_{L^2_{\xi,\tau}}
  &\lesssim \left| \int_0^1  |\mu|^4 \left| \hat{\chi g}(\mu^3)  \right| \d \mu  \right| \\
  &\approx \left| \int_0^1  |z|^2 \left| \hat{\chi g}(z)  \right| \d z \right| \lesssim \| \chi g \|_{L^2}.
\end{align*}
Since ${2s+6b-1} \geq 0$, this estimate implies the desired $H^{\frac{2s + 6b -1}{3}}$ bound. This completes the proof.
\end{proof}

The following Sobolev space estimates for the linear term $W_0^t(0,g)$ are required as well.
\begin{lemma}\label{linHs} 
 For any $s \geq -1$ and $g$ such that $\chi g \in H^{\frac{s+1}{3}}(\R)$, we have 
 \begin{align*}
  W^t_0(0,g) \in C^0_tH^s_x(\R \times \R) \quad \text{and} \quad
  \eta(t) W^t_0(0,g) \in C^0_x H^{\frac{s+1}{3}}(\R \times \R). 
 \end{align*}
\end{lemma}

\begin{proof}
Recall the formula \eqref{eq:IBVPsol} for $W^t_0$. First, we establish that $W_0^t(0,g)$ is in $C^0_tH^s_x$. Again let $f(y) =  e^{-i y \cos(\pi/3)} e^{ - y \sin(\pi/3)} \rho(y)$. Then we have
\begin{multline*} 
 W_0^t(0,g(x)) 
= 
\frac3\pi \Re \int_0^\infty e^{i \mu^3 t} f(\mu x) \;  \hat{\chi g}(\mu^3)  \mu^2 \d \mu 
\\= 
\frac3\pi \Re \int_{-\infty}^\infty  f(\mu x) \;  \mathcal{F}\Bigl( W_\R^t \psi \Bigr)(\mu)\d \mu,
\end{multline*}
where 
\[ \hat{\psi}(\mu) = \chi(\mu) \mu^2 \hat{\chi g}(\mu^3) . \]
Now
\begin{align*}
 \| \psi \|_{H^s(\R)}^2 
 &=
 \int_0^\infty \left| \hat{\chi g}(\mu^3)  \right|^2 \mu^4 \lb \mu \rb^{2s} \d \mu \\
 &= 
 \frac13 \int_0^\infty \left| \hat{\chi g}(z)  \right|^2 |z|^{2/3} \lb z \rb^{2s/3} \d z 
 \lesssim \| \chi g \|_{H^\frac{s+1}3(\R)}^2
\end{align*}
for any $s$. Note that for $-1 \leq s \leq 0$ we also have $\| \psi \|_{\dot{H}^s} \lesssim \| \chi g \|_{{H}^\frac{s+1}3(\R)}$. 

Thus, to show that $W_0^t(0,g)$ is in $C^0_tH^s_x$, by continuity of the operator $W_\R^t$, it suffices to show that the operator $T: h \mapsto Th$ given by 
\[ Th(x) = \int_{-\infty}^\infty f(\mu x) \hat{h}(\mu) \d \mu \]
is continuous from $H^s_x$ to $H^s_x$ for $s \geq 0$, and from $\dot{H}^s_x$ to $H^s_x$ for $-1 \leq s \leq 0$. Write 
\[ |Th(x)| \lesssim \int_{-\infty}^\infty |f(z)| \left|\hat{h}(z/x) \right| \; \frac1{|x|} \d \mu. \]
Then 
\begin{multline*}
 \| Th(x)\|_{L^2_x} \lesssim \int_{-\infty}^\infty |f(z)| \left\| \hat{h}(z/x)  \; \frac1{|x|} \right\|_{L^2_x} \d \mu 
  \\\approx 
 \| h \|_{L^2} \int_{-\infty}^\infty |f(z)|/\sqrt{|z|} \d \mu 
 \lesssim \| h \|_{L^2}
\end{multline*}
since $f$ is a Schwartz function.
This establishes the continuity of $T$ for $s = 0$. For $s \in \N$, we note that $\partial_x^s \left[ Th(x)\right] = \int_{-\infty}^\infty f^{(s)}(\mu x) \mu^s \hat{h}(\mu) \d \mu$, and argue as above. Interpolation completes the argument for positive $s$.

For negative $s$, note that 
\[\left| \lb \partial_x \rb^{-1} \left[ Th(x)\right] \right| \lesssim \left|  \partial_x ^{-1} \left[ Th(x)\right] \right| \lesssim \left| \int_{-\infty}^\infty [\partial^{-1}f](\mu x) \mu^{-1} \hat{h}(\mu) \d \mu \right|. \]
Using the same argument as in the $s=0$ case, we conclude that we have $\| Th(x)\|_{H^{-1}_x} \lesssim \| h \|_{\dot{H}^{-1}}$ as desired. 

It remains to show that $\eta(t) W^t_0(0,g)$ is in $C^0_xH^{\frac{s+1}3}_t$. Write 
\[ \eta(t) W_0^t(0,g) = \eta(t) \int_{-\infty}^\infty \mathcal{F}_\mu \left( {f}(x\mu) \right)(y) \;  W^t_\R \psi(y) \d y, \]
where $\psi$ is defined as above and we have used the identity $\int p \hat{q} = \int \hat{p} q$. This is equivalent to
\[  \int_{-\infty}^\infty  \hat{f}(y/x)  \; \Big[  \eta(t) W_\R^t \psi\Big](y) \frac{\d y}{x} = \int_{-\infty}^\infty  \hat{f}(z) \; \Big[  \eta(t) W_\R^t\psi\Big](xz) \d z. \]
Using Kato smoothing, Lemma \ref{kato}, with the fact that $\hat{f} \in L^1$, the proof is completed.  
\end{proof}

\subsection{Nonlinear Estimates}
In this section, we state various nonlinear estimates which will be required to complete our proof. In these estimates, one should think of $b = \frac12 +$ and $b' = \frac12-$. This is the most important case for our arguments.

The first nonlinear estimate is an $X^{s,b}$ space estimate for the 3-gKdV nonlinearity. The proof, which is based on arguments from \cite{G}, is in Section \ref{nonlinProof}. Notice that the smoothing effect vanishes as $s \searrow -\frac16$ for $b' = \frac12-$. 
\begin{proposition}\label{nonlin}
Fix $s > -\frac16$, with $b > \frac12$ and $b' \in (\frac{5-6s}{12}, \frac12)$. Let $0 \leq a < \frac12(3\min\{0,s\} + 6b' - \frac52)$. Then we have 
\begin{equation*} 
\| \partial_x(u_1 u_2 u_3 u_4) \|_{X^{s +a, -b'}} \lesssim \prod_{j=1}^4 \| u_j \|_{X^{s,b}} .
\end{equation*}
\end{proposition}

We also need to show that the nonlinear term is in an appropriate temporal Sobolev space. This is established via the following estimate on the Duhamel integral term. The proof is in Section \ref{duhamelEstProof}.
\begin{proposition} \label{duhamelEst}
 For $b' \in (0, \frac12)$, we have 
\begin{multline*}
  \left\| \eta(t) \int_0^t W^{t-t'}_\R [N(x,t')] \d t'  \right\|_{L^\infty_x H^\frac{s+1}{3}_t} 
\lesssim  \\
\begin{cases}
\| N \|_{X^{s,-b'}} &\text{ if } -1 \leq s \leq 2 - 3b' \\
\| N \|_{X^{s,-b'}} + \left\| \int \chi_R(\xi, \tau) \lb \tau - \xi^3 \rb^{\frac{s-2}{3}} | \hat{N}(\xi,\tau)| \d \xi \right\|_{L^2_\tau} &\text{ if }  s > 2-3b',
\end{cases}
\end{multline*}
where $R$ is the set $R = \bigl\{ (\xi,\tau) \; : \; |\tau| \geq 1 \text{ and } |\tau | \gg |\xi|^3\bigr\}$. 
\end{proposition}

Finally, we'll require an estimate for the correction term which appears in the above proposition. The proof of this statement is found in Section \ref{correctionEstProof}. 

\begin{proposition} \label{correctionEst}
 Let $N(x,t) = \partial_x \left( u_1 u_2 u_3 u_4 \right)$, and define the set $R$ as in the previous proposition: $R = \bigl\{ (\xi,\tau) \; : \; |\tau| \geq 1 \text{ and } |\tau | \gg |\xi|^3\bigr\}$. Then for $b > \frac12$ and $0 \leq a < \min\{ \frac12, 2-s\}$, we have
 \[ \left\| \int \chi_R(\xi, \tau) \lb \tau - \xi^3 \rb^{\frac{s+a-2}{3}} | \hat{N}(\xi,\tau)| \d \xi \right\|_{L^2_{\tau}} \lesssim \prod_{j=1}^4 \| u_j\|_{X^{s,b}}. \]
\end{proposition}

\section{Local Theory} \label{locTheory}

We wish to find a solution to the nonlinear gKdV-3 problem
\begin{equation*}
\begin{cases}
u_t + u_{xxx} + (u^4)_x = 0, \quad (x,t) \in \R^+ \times \R^+, \\
u(x, 0) = u_0(x) \in H^{s}(\R^+), \\
u(0,t) = g(t) \in H^\frac{s+1}{3}(\R^+). 
\end{cases}
\end{equation*}
Begin by solving the linear equation
\begin{equation*}
\begin{cases}
L_t + L_{xxx}  = 0, \quad (x,t) \in \R^+ \times \R^+, \\
L(x, 0) = u_0(x) \in H^{s}(\R^+), \\
L(0,t) = g(t) \in H^\frac{s+1}{3}(\R^+).
\end{cases}
\end{equation*}
We obtain the solution $L$ for $t \in [0,1]$ by using the forumula \eqref{eq:fullairy}. 
Thus $ L$ lies in the space 
\[ C^0_tH^s_x \cap C^0_x H^\frac{s+1}3_t \cap X^{s, \frac12} \cap X^{s-3\epsilon, \frac12 + \epsilon}\]
for $\epsilon > 0$. To see this, note that $\eta(t) W^t_\R u_0(x) \in X^{s,b}$ for any $b$ by Lemma \ref{RlinXsb} and $\eta(t) W^t_\R u_0(x) \in C^0_tH^s_x \cap C^0_x H^\frac{s+1}3_t$ by Lemma \ref{kato}. We also have by Lemma \ref{kato} that $\eta(t)\bigl[ W_\R^t u_0(x)\bigr]_{x=0} \in H^\frac{s+1}3_t$. Hence the $W^t_0$ term above is in $X^{s - 3\epsilon, \frac12 + \epsilon}$ for any $\epsilon \geq 0$ by Lemma \ref{linXsb}, and in $ C^0_tH^s_x \cap C^0_x H^\frac{s+1}3_t$ by Lemma \ref{linHs}. It is important to note that we can put this linear flow in a Fourier restriction space $X^{s_0,b}$ with $b > 1/2$ as long as we are willing to accept some loss of regularity; i.e. if we can work with $s_0 < s$. The nonlinear smoothing effect will enable us to close the argument despite this small regularity loss.

Now let $v := u-L$, so that $v$ solves the difference equation
\begin{equation*}
\begin{cases}
v_t + v_{xxx} + \partial_x[(v+L)^4] = 0, \quad (x,t) \in \R^+ \times \R^+, \\
v(x, 0) = 0 \in H^{s}(\R^+), \\
v(0,t) = 0 \in H^\frac{s+1}{3}(\R^+). 
\end{cases}
\end{equation*}
We will show that a solution $v$ to this difference equation exists, which will imply existence of solution $v + L$ to the original 3--gKdV. We proceed with a contraction argument in the space $X^{s+a, \frac12 + \epsilon}$. 

The Duhamel formulation is as follows. A solution $v$ satisfies $v = \Phi(v)$ on a time interval $[0,T]$, with $0< T <1$, where 
\begin{equation} \label{eq:duhamelForm}
\Phi \bigl(v(x,t)\bigr) = \eta_T(t) \int_0^t W_\R^{t-t'}  \bigl[G(v,L)\bigr] \d t' - \eta_T(t) W^t_0(0,h),
\end{equation}
with
\begin{equation}\label{eq:nonlinearity} 
\begin{split}
&G(v,L) = \eta_T(t) \partial_x[(v+L)^4] 
\\
&h(t) = \eta_T(t) \left[ \int_0^t W_\R^{t-t'} \bigl[G(v,L)\bigr] \d t'  \right]_{x=0}.\end{split}
\end{equation}

Suppose first that $s + a \leq \frac12$. Using Lemma \ref{RduhamelXsb} and Lemma \ref{linXsb}, and then Lemma \ref{powerT} and Lemma \ref{duhamelEstProof}, and finally Lemma \ref{powerT} again, we obtain
\begin{align*}
 \| \Phi(v) \|_{X^{s+a, \frac12 + \epsilon}} & \lesssim \| G(v,L) \|_{X^{s+a,-\frac12 + \epsilon}} + \| h \|_{H^{\frac{s+a+1}3 + \epsilon}} \\
 &\lesssim T^\epsilon \| G(v,L) \|_{X^{s+a,-\frac12 + 2\epsilon}} + \| G(v,L)\|_{X^{s+a+3\epsilon, -\frac12 + \epsilon}} \\
 &\lesssim T^\epsilon \| G(v,L)\|_{X^{s + a + 3\epsilon, -\frac12 + 2 \epsilon}}.
\end{align*}
We now apply Lemma \ref{nonlin} with $b' = \frac12 - 2 \epsilon$, $\tilde{s} = s - 3\epsilon$, and $\tilde{a} = a + 6 \epsilon$ to conclude that 
\begin{align*}
 \| G(v,L)\|_{X^{s + a + 3\epsilon, -\frac12 + 2 \epsilon}} &\lesssim \left( \| v \|_{X^{s-3\epsilon, \frac12 +\epsilon}} + \|L \|_{X^{s - 3 \epsilon, \frac12 + \epsilon}} \right)^4  
 \\
 &\lesssim \left( \| v \|_{X^{s, \frac12 +\epsilon}} + \|u_0\|_{H^s_x} + \|g\|_{H^{\frac{s+1}3}_t} \right)^4 .\end{align*}
The hypotheses of Lemma \ref{nonlin} are satisfied as long as $s > -\frac16 + 7\epsilon$ and $a < \frac14 + \frac32\min\{0,s\} -\frac{33}2\epsilon \leq \frac14 + \frac32 \min\{0,s-3\epsilon\} - 12 \epsilon$. Thus if we choose $T$ sufficiently small, we may close the contraction on a time interval $[0,T]$. We may also conclude that $v \in C^0_xH^\frac{s+a+1}3_t$ by Lemmata \ref{duhamelEst}, \ref{nonlin}, and \ref{linHs}. So we obtain a solution 
\[ u = L+ v  \in   
\Bigl[ C^0_tH^s_x \cap C^0_x H^\frac{s+1}3_t \cap X^{s, \frac12} \Bigr]
+
\Bigl[ C^0_tH^{s+a}_x \cap C^0_x H^\frac{s+a+1}3_t \cap X^{s+a, \frac12+} \Bigr] 
. \]

For $\frac12 < s < 2$, the argument is essentially the same. However, when estimating the Duhamel term using Proposition \ref{duhamelEst}, the $L^2_\tau$ correction term appears. This is controlled using Proposition \ref{correctionEst} and we proceed as above. Notice that the total solution $u$ satisfies
\[ u \in  \Bigl[ C^0_tH^s_x \cap C^0_x H^\frac{s+1}3_t \cap X^{s, \frac12} \Bigr]. \]

\section{Uniqueness} \label{uniq}

In this section, we discuss uniqueness of solutions to the 3-gKdV initial boundary value problem. We first consider the case of relatively smooth data. 
Suppose $u_1$ and $u_2$ are two solutions to the 3-gKdV \eqref{eq:3gkdv} with the same initial and boundary data. Then the difference $u := u_1 - u_2$ satisfies the equation
\[ u_t + u_{xxx} + (u_1^4 - u_2^4)_x = 0,\]
with zero initial and boundary data. 
Then, using integration by parts and the zero boundary condition, we have
\begin{align*}
 \partial_t \| u(\cdot, t) \|_{L^2(\R^+)}^2 &= - 2 \int_0^\infty uu_{xxx} + u(u_1^4 - u_2^4)_x \d x \\
 &= 2\int_0^\infty u_xu_{xx} - u (u_1^4 - u_2^4)_x \d x \\
 &= - u_x(0,t)^2 -  2\int_0^\infty  u (u_1^4 - u_2^4)_x \d x  \\
 &\leq -  2\int_0^\infty  u (u_1^4 - u_2^4)_x \d x \\
 &= -2\int_0^\infty  u \Bigl[u(u_1^3 + u_1^2 u_2 + u_1u_2^2 + u_2^3)\Bigr]_x \d x \\
 &= -\int_0^\infty  u^2 \Bigl[u_1^3 + u_1^2 u_2 + u_1u_2^2 + u_2^3\Bigr]_x \d x.
\end{align*}
Thus by Sobolev embedding we see that $\partial_t \| u(\cdot, t) \|_{L^2(\R^+)}^2$ can be bounded by 
\begin{multline*} 
  \| u \|^2_{L^2(\R^+)} \Bigl(\|u_1\|_{H^{\frac32+}(\R^+)} \| u_1\|_{H^{\frac12+}(\R^+)}^2 + \|u_2\|_{H^{\frac32+}(\R^+)} \| u_2\|_{H^{\frac12+}(\R^+)}^2 \Bigr) \\
 \lesssim \| u \|^2_{L^2(\R^+)}, 
\end{multline*}
where the last inequality holds if $u_1$ and $u_2$ are bounded in $H^{\frac32+}$. Gr\"onwall's inequality then implies that $\| u(\cdot, t)\|_{L^2} = 0$ for all $t$, and hence $u_1 = u_2$. This concludes the proof of uniqueness for solutions in $H^{\frac32+}_x$. 

It remains to consider uniqueness for rougher solutions. For that we use a variant of the argument in \cite{CT}. Suppose we have  initial and boundary data
\[ (u_0, g) \in  H^{s_0}(\R^+) \times H^{\frac{s_0+1}{3}}(\R^+).\]
Suppose first that $s_0 \in (\frac54, \frac32)$. In addition suppose $u_{0,e}$ and $\wt{u_{0,e}}$ are two $H^{s_0}(\R)$ extensions of $u_0$. 

Let $u$ and $\wt{u}$ be the corresponding solutions to the fixed point equation. We wish to show that $u$ and $\wt{u}$ are equal on $\R^+_x$, at least for some short time.

Take a sequence ${u_{0,k}}$ in $H^{\frac32+}(\R^+)$ which converges to $u_0$ in $H^{s_0}(\R^+)$. Let $u_{0,e,k}$ and $\wt{u_{0,e,k}}$ be $H^{\frac32+}(\R)$ extensions of $u_{0,k}$ which converge to $u_{0,e}$ and $\wt{u_{0,e}}$ respectively in $H^{\frac12-}(\R)$. Such extensions exist by Lemma \ref{approx_lemma} below.

Using the local theory in $H^{\frac32+}$, we arrive at corresponding sequences of solutions $u_k$ and $\wt{u_k}$. Since their initial data is equal on the right half line, the uniqueness result above implies that $u_k$ and $\wt{u_k}$ are equal on $\R^+$ on their common interval of existence. Furthermore, $u_k$ converges to $u$ and $\wt{u_k}$ converges to $\wt{u}$ in $H^{\frac12-}$ as $k$ increases, by the local well--posedness result we established in Section \ref{locTheory}. Thus, if the common interval of existence does not vanish as $k$ increases, we will have uniqueness. 

\emph{A priori}, the interval of existence is inversely proportional to the $H^{\frac32+}$ norm of the initial data (as well as the norm of the boundary data). This norm is growing as $k$ increases. This means that the time of existence goes to zero as $k \to \infty$. However, using the smoothing, we can take the time of existence proportional to the data in the $H^{s_0}$ norm, which is bounded as desired. This works directly for $s_0 > \frac54$. Iterating the argument, we obtain uniqueness for $s_0 > -\frac16$.

\begin{lemma}{\cite{ET1}}\label{approx_lemma}
 Fix $-\frac12 < s_0 < \frac12$ and $k > s_0$. Let $u_0 \in H^{s_0}(\R^+)$ and $f \in H^k(\R^+)$. Let $u_{0,e}$ be an $H^{s_0}$ extension of $u_0$ to $\R$. Then there is an $H^k$ extension $f_e$ of $f$ to $\R$ such that 
 \[ \| u_{0,e} - f_e \|_{H^r(\R)} \lesssim \| u_0-f\|_{H^s(\R^+)} \quad \text{ for } r < s_0. \] 
\end{lemma}

\section{Proofs} \label{proofs}

Before proceeding, we state a calculus lemma which will useful. For proofs of similar results, see \cite{ET2}.
\begin{lemma}\label{calcEst}
 If $\alpha > 1$ and $\alpha \geq \beta \geq 0 $, then  
\[\int_\mathbb{R} \lb y - a \rb^{-\alpha} \lb y-b \rb^{-\beta} \d y \lesssim \lb a - b\rb ^{- \beta}. \] 
\end{lemma}

\subsection{Proof of Proposition \ref{nonlin}}\label{nonlinProof}
The following results are used in the proof, and are placed here for convenient reference.
\begin{lemma} We have the following estimates:
\begin{align} \label{eq:trivial}
\| u \|_{L^2_{x,t}} 
&= \| u \|_{X^{0,0}} ,  
& 
\text{(Trivial $L^2_{x,t}$ estimate)}, & \hspace{.5in} \\
\label{eq:L8}
\| u\|_{L^8_{x,t}} 
&\lesssim \| u \|_{X^{0,\frac12+}},  
&  
\text{\cite[Theorem 2.4]{KPV91}},  \\
\label{eq:G}
\| D^{\frac1r} u \|_{L^r_t L^{\frac{2r}{r-4}}_x} &\lesssim \| u \|_{X^{0, \frac12+}}, \quad r \geq 4,
&
\text{\cite[Lemma 2(i)]{G}}, \\
\label{eq:Sobolev}
\| u \|_{ L^\infty_x} &\lesssim \|D^{\frac1q +} u \|_{L^q_x}, \quad q \geq 2,
&
\text{(Sobolev embedding)}.
\end{align}
\end{lemma}

We note that one can upgrade the linear Strichartz estimates into a priori $X^{s,b}$ estimates of the form above, using the definition of the $X^{s,b}$ spaces and Cauchy-Schwarz inequality, see \cite{ETbook}.
\vskip 0.05in
\noindent
We now proceed with the proof of the proposition. 

	Let $f_j(\xi, \tau) = \lb \xi \rb^s \lb \tau - \xi^3 \rb^{b} \hat{u_j}(\xi,\tau)$. Then the desired estimate is  
	\begin{multline} \label{eq:nonlin2}
	\left\| \frac{ | \xi_0 | \lb \xi_0 \rb^{s + a}}{\lb \tau_0 - \xi_0^3 \rb^{b'}}   \underset{{\substack{\sum \tau_j = 0\\\sum \xi_j =0}}}{\iiint\!\!\iiint} \prod_{j=1}^4 \frac{f_j(\xi_j, \tau_j)}{\lb \xi_j \rb^{s} \lb \tau_j - \xi_j^3 \rb^{b} }
	\d \xi_1 \d \xi_2 \d \xi_3  \d \tau_1  \d \tau_2 \d \tau_3\right\|_{L^2_{\xi_0,\tau_0}} 
	\\ \lesssim \prod_{j=1}^4 \| f_j \|_{L^2_{\xi,\tau}}. 
	\end{multline}	
The notation above means that the inner integral is taken over the surfaces $\sum_{j=0}^4 \tau_j = 0$ and $\sum_{j=0}^4 \xi_j =0$. For fixed $(\xi_0,\tau_0)$, this is a six-dimensional surface in frequency space.
	
	We consider several cases. In the following, we write $|\xi_{min}| = \min_{j=1}^4 |\xi_j|$ and $|\xi_{max}|  = \max_{j=1}^4 |\xi_j|$. 
	\quad \\
        
	\noindent\textbf{\underline{Case 0: $|\xi_j| \lesssim 1$ for all $j\in\{0,1,2,3,4\}$.}}
	\, By a standard argument invoking the Cauchy--Schwartz inequality and Young's inequality, \eqref{eq:nonlin2} can be reduced to showing that 
	\[ \sup_{\xi_0, \tau_0} \frac{\lb \xi_0 \rb^{2s +2 + 2a}}{\lb \tau_0 - \xi_0^3 \rb^{2b'}} \iiint\limits_{\sum \tau_j = 0} \; \; \iiint\limits_{\sum \xi_j = 0}   \frac{\d \xi_1 \d \xi_2 \d \xi_3 \d \tau_1 \d \tau_2 \d \tau_3}{\prod_{j=1}^4\lb \xi_j \rb^{2s} \lb \tau_j - \xi_j^3 \rb^{2b} }
	< \infty.\]
	In the current case, the $\xi_j$ weights may be disregarded and we are left to bound
	\[ \sup_{|\xi_0| \lesssim 1, \; \tau_0}  \iiint\limits  \iiint\limits_{| \xi_j| \lesssim 1}   \frac{ \prod_{j=1}^3\lb \tau_j - \xi_j^3 \rb^{-2b}  \quad   \d \xi_1 \d \xi_2 \d \xi_3 \d \tau_1 \d \tau_2 \d \tau_3}{\lb \tau_0 - \xi_0^3 \rb^{ 2b'}  \lb (\tau_0 + \tau_1 + \tau_2 + \tau_3) - (\xi_0+\xi_1 + \xi_2 +\xi_3)^3 \rb^{2b} }.\]
	Integrating in $\tau_j$ repeatedly and then using the inequality $\lb a + b \rb \lesssim \lb a \rb \lb b \rb$, we arrive at 
	\[ \sup_{|\xi_0| \lesssim 1} \;  \iiint\limits_{| \xi_j| \lesssim 1}   \frac{  \d \xi_1 \d \xi_2 \d \xi_3 }{\lb \xi_0^3 + \xi_1^3 + \xi_2^3 + \xi_3^3 - (\xi_0+\xi_1 + \xi_2 +\xi_3)^3 \rb^{2b' } }, \]
	which is finite. 	
	\quad \\

	\noindent\textbf{\underline{Case 1: $ |\xi_{min}| \geq \frac{99}{100} |\xi_{max}|$ \& $|  \sum_{j=0}^4 \xi_1^3 | \approx |\xi_1|^3$.}}\,
	In this case, it will be helpful to define the maximum modulation 
	$M := \max\left\{ \max_{j=0}^4 \lb \tau_j - \xi_j^3\rb \right\}$. Note that $M \gtrsim \left|  \sum_{j=0}^4 \xi_j^3 \right|$. Then in this case we have 
	\[ | \xi_0| \lb \xi_0 \rb^{s +a} \prod_{j=1}^4 \lb \xi_j \rb^{-s} \lesssim M^{\frac13 -s + \frac{a}3}. \]
	Hence, writing \eqref{eq:nonlin2} in its dual form, we wish to establish that
	\begin{multline*} 
	\left|    \iiiint\limits_{\sum \tau_j = 0}  \iiiint\limits_{\sum \xi_j = 0} \frac{M^{\frac13 - s + \frac{a}3}}{\lb \tau_0 - \xi_0^3 \rb^{b'-b}}   \prod_{j=0}^4 \frac{f_j(\xi_j, \tau_j)}{ \lb \tau_j - \xi_j^3 \rb^{b} }
	\d \xi_1 \d \xi_2 \d \xi_3 \d \xi_4 \d \tau_1 \d \tau_2 \d \tau_3 \d \tau_4 \right| \\
	\lesssim  \prod_{j=0}^4 \| f_j \|_{L^2_{\xi,\tau}}. 
	\end{multline*}
	Suppose first that $M = \lb \tau_0 - \xi_0^3 \rb$. Then we can bound the left-hand side of the above quantity by
	\begin{align*} \label{eq:nonlin2}
	&\left|    \iiiint\limits_{\sum \tau_j = 0}  \iiiint\limits_{\sum \xi_j = 0} \frac{f_0(\xi_0, \tau_0)}{\lb \tau_0 - \xi_0^3 \rb^{b'-\frac13 + s - \frac{a}3}}   \prod_{j=1}^4 \frac{f_j(\xi_j, \tau_j)}{ \lb \tau_j - \xi_j^3 \rb^{b} }
	\d \xi_j \d \tau_j \right| \\
	\lesssim &\left|    \iiiint\limits_{\sum \tau_j = 0}  \iiiint\limits_{\sum \xi_j = 0} {f_0(\xi_0, \tau_0)}  \prod_{j=1}^4 \frac{f_j(\xi_j, \tau_j)}{ \lb \tau_j - \xi_j^3 \rb^{b + \frac14 b' - \frac1{12} + \frac{s}4 - \frac{a}{12}} }
	\d \xi_j \d \tau_j \right| \\
	\lesssim &\| f_0 \|_{L^2_{\xi,\tau}} \left\|  \;\; \iiint\limits_{\sum \tau_j = 0}  \;\; \iiint\limits_{\sum \xi_j = 0} \;    \frac{\prod_{j=1}^4f_j(\xi_j, \tau_j)\d \xi_1 \d \xi_2 \d \xi_3 \d \tau_1 \d \tau_2 \d \tau_3}{\prod_{j=1}^4 \lb \tau_j - \xi_j^3 \rb^{b + \frac14 b' - \frac1{12} + \frac{s}4 - \frac{a}{12}} }
	 \right\|_{L^2_{\xi_0,\tau_0}} \\
	\lesssim &\| f_0 \|_{L^2_{\xi,\tau}} \; \prod_{j=1}^4 \left\| \mathcal{F}^{-1} \left( \frac{f_j}{ \lb \tau - \xi^3 \rb^{b + \frac14 b' - \frac1{12} + \frac{s}4 - \frac{a}{12}} } \right)  \right\|_{L^{8}_{x,t}} \lesssim  \prod_{j=0}^4 \| f_j\|_{L^2_{\xi,\tau}} . 
	\end{align*}
	The last inequality uses the $X^{0,\frac12+} \hookrightarrow L^8_{x,t}$ embedding \eqref{eq:L8}, which is applicable as long as $a < 3s + 12b + 3b' - 7$. If the maximum modulation is $\lb \tau_j - \xi_j^3 \rb$ for some $j = 1,2,3,4$, the same argument applies. We distribute the excess power of $M$ to the other modulation multipliers. Each again ends up with exponent $b + \frac14 b' - \frac1{12} + \frac{s}4 - \frac{a}{12}$, and we argue as above.  
	\quad \\

	\noindent\textbf{\underline{Case 2: $ |\xi_{min}| \geq \frac{99}{100} |\xi_{max}|$ \& two $\xi_j$ are positive, or $|\xi_{min}| < \frac{99}{100} |\xi_{max}|$.}}\,  
	We may assume by relabeling that frequencies are numbered in descending order, i.e. 
	\[ | \xi_1| \geq |\xi_2| \geq |\xi_3| \geq |\xi_4|. \]

	\noindent\textbf{\underline{Case 2a: $|\xi_2| \approx |\xi_1|$.}}\,  
	In this case, $|\xi_0| \lesssim | \xi_1| \approx | \xi_2|$. We may thus write, for $s \leq 0$,
	\begin{equation} \label{eq:multBound}
	| \xi_0| \lb \xi_0 \rb^{s+a} \prod_{j=1}^4 \lb \xi_j \rb^{-s} \lesssim | \xi_{1} - \xi_{4}|^{\frac12 } | \xi_{1} + \xi_{4}|^{\frac12 } \lb \xi_{2} \rb^{\frac{-3s}2 + a } \lb \xi_{3} \rb^{\frac{ -3s}2},
	\end{equation}
or equivalently
	\begin{equation*}
	| \xi_0| \lb \xi_0 \rb^{s+a} \lesssim | \xi_{1} - \xi_{4}|^{\frac12 } | \xi_{1} + \xi_{4}|^{\frac12 } \lb \xi_1 \rb^s \lb \xi_4 \rb^s \lb \xi_{2} \rb^{-\frac{s}2 + a } \lb \xi_{3} \rb^{-\frac{ s}2},
	\end{equation*}	
	
	Using the notation defined in \eqref{eq:operatornames}, it suffices to bound
	\begin{align*}
	& \left\| \left( \I^{\frac12} \I_-^{\frac12} (D^s u_1 D^s u_4 ) \right)  \left( D^{a- \frac{s}{2}} u_2 \right)  \left( D^{-\frac{s}{2} } u_3 \right)   \right\|_{X^{0,-b'}} \\
	\lesssim & 
	\left\| \left( \I^{\frac12 } \I_-^{\frac12 } (D^s u_1 D^s u_4 ) \right)  \left( D^{a-\frac{s}{2} } u_2 \right)  \left( D^{-\frac{s}{2} } u_3 \right)   \right\|_{L_t^{\frac2{1+2b'}+}L^2_x} \\
	\leq &
	\left\|  \I^{\frac12} \I_-^{\frac12} (D^s u_1 D^s u_4 ) \right\|_{L^2_{x,t}}  \left\| D^{a-\frac{s}{2} } u_2 \right\|_{L^{\frac2{b'}+}_tL^\infty_x}   \left\| D^{-\frac{s}{2}} u_3   \right\|_{L^{\frac{2}{b'}+}_tL^\infty_x}.
	\end{align*}

	Using \cite[Corollary 1]{G}, we may bound the $L^2_{x,t}$ norm in the previous line by $\| u_1 \|_{X^{s,b}} \| u_4 \|_{X^{s,b}}$. 
	It remains to show that 
	\[  \left\| D^{a-\frac{s}{2}} u   \right\|_{L^{\frac{2}{b'}+}_tL^\infty_x} \lesssim \| u \|_{X^{s,b}}, \!\quad \text{or equivalently} \!\quad  \left\| D^{a-\frac{3s}{2} } u   \right\|_{L^{\frac{2}{b'}+}_tL^\infty_x} \lesssim \| u \|_{X^{0,b}}.\] 	
	We have by Sobolev embedding \eqref{eq:Sobolev} 
	\begin{align*} \label{eq:sob}
	\left\| D^{a-\frac{3s}{2} } u \right\|_{L^{\frac{2}{b'}+}_tL^\infty_x} \lesssim
	\left\| D^{a-\frac{3s}{2} + \frac1q + } u \right\|_{L^{\frac{2}{b'}+}_tL^q_x}
	\end{align*}
	for any $q \in [2,\infty]$. The $X^{0, \frac12+} \hookrightarrow L^{r}_t H^{\frac1{r}, \frac{2r}{r-4}}_x$ estimate \eqref{eq:G}, which holds for $r \geq 4$, gives the desired bound as long as $a < \frac12(3s + 3b' - 1)$. 

If $s >0$, replace \eqref{eq:multBound} by 
\begin{equation*} 
	| \xi_0| \lb \xi_0 \rb^{s+a} \prod_{j=1}^4 \lb \xi_j \rb^{-s} \lesssim | \xi_{1} - \xi_{4}|^{\frac12 } | \xi_{1} + \xi_{4}|^{\frac12 } \lb \xi_{2} \rb^{a }.
	\end{equation*}
The argument then carries through as above as long as $a < \frac12(3b' - 1)$.

	\quad \\

	\noindent\textbf{\underline{Case 2b: $ |\xi_{2}| \ll |\xi_{1}|$.}}\, That is, $|\xi_0| \approx |\xi_1| \gg |\xi_2|, |\xi_3|, |\xi_4|$.  In this case we have for $s \leq 0$ the inequality
	\begin{equation} \label{eq:multBound2}
	| \xi_0| \lb \xi_0 \rb^{s+a} \prod_{j=1}^4 \lb \xi_j \rb^{-s} \lesssim | \xi_{4} - \xi_{1}|^{\frac12 } | \xi_{1} + \xi_{4}|^{\frac12 } \lb \xi_{0} \rb^{\frac{-3s}2 + a } \lb \xi_{2} \rb^{\frac{ -3s}2}.
	\end{equation}
	Using the dual formulation with the bound \eqref{eq:multBound2}, it would suffice to show that 
	\[ \left\| \left( \I^{\frac12} \I_-^{\frac12} (D^s u_1 D^s u_4 ) \right)  \left( D^{a- \frac{s}{2}} u_0 \right)  \left( D^{-\frac{s}{2} } u_2 \right)   \right\|_{X^{0,-b}}  \!\!\lesssim \| u_0 \|_{X^{s,b'}} \prod_{j=1,2,4} \| u_j\|_{X^{s,b}}.\] 
	We have 
	\begin{align*}
	& \left\| \left( \I^{\frac12} \I_-^{\frac12} (D^s u_1 D^s u_4 ) \right)  \left( D^{a- \frac{s}{2}} u_0 \right)  \left( D^{-\frac{s}{2} } u_2 \right)   \right\|_{X^{0,-b}} \\
	\lesssim & 
	\left\| \left( \I^{\frac12 } \I_-^{\frac12 } (D^s u_1 D^s u_4 ) \right)  \left( D^{a-\frac{s}{2} } u_0 \right)  \left( D^{-\frac{s}{2} } u_2 \right)   \right\|_{L_t^1 L^2_x} \\
	\leq &
	\left\|  \I^{\frac12} \I_-^{\frac12} (D^s u_1 D^s u_4 ) \right\|_{L^2_{x,t}}  \left\| D^{a-\frac{s}{2} } u_0 \right\|_{L^4_tL^\infty_x}   \left\| D^{-\frac{s}{2}} u_2   \right\|_{L^4_tL^\infty_x}.
	\end{align*}
	So it remains to show that 
	\[ \left\| D^{a-\frac{3s}{2} } u   \right\|_{L^4_tL^\infty_x} \lesssim \| u \|_{X^{0,b'}}.\] 
	Again we have by Sobolev embedding \eqref{eq:Sobolev} 
	\begin{align*} 
	\left\| D^{a-\frac{3s}{2} } u \right\|_{L^4_tL^\infty_x} \lesssim
	\left\| D^{a-\frac{3s}{2} + \frac1q + } u \right\|_{L^4_tL^q_x}
	\end{align*}
	for any $q \in [2,\infty]$.
	Interpolation between the $L^{r}_t H^{\frac1{r}, \frac{2r}{r-4}}_x$ estimate \eqref{eq:G} and the trivial $L^2_{x,t}$ estimate \eqref{eq:trivial} gives, for $0 \leq \theta \leq 1$ and $r \geq 4$, 
	\[ \| D^{\frac\theta{r}} u \|_{L^\frac{2r}{2\theta + r(1-\theta)}_t L^\frac{2r}{r - 4\theta}_x} \lesssim \| u \|_{X^{0,\frac\theta2+}} .\]
	For any $0 < \epsilon \ll 1$, if we take 
	\[ \theta = 2(b' - \epsilon), \quad r = 2 + \frac{2}{4(b' - \epsilon) - 1} > 4, \quad \text{and} \quad q = \frac1{1-2(b'-\epsilon)},\]
	then the interpolation estimate gives 
	\[ \left\| D^{a-\frac{3s}{2} + \frac1q + } u \right\|_{L^4_tL^q_x} \lesssim \| u \|_{X^{0,b'}} \]
	as long as 
	\[ a - \frac{3s}2 + \frac1q + \leq \frac{\theta}{r}, \quad \text{ i.e. } \quad a < \frac12\left(3s + 6b' - \frac52\right).\] 
	This is the worst case.

If $s >0$, we may replace \eqref{eq:multBound2} by 
\begin{equation*} 
	| \xi_0| \lb \xi_0 \rb^{s+a} \prod_{j=1}^4 \lb \xi_j \rb^{-s} \lesssim | \xi_{4} - \xi_{1}|^{\frac12 } | \xi_{1} + \xi_{4}|^{\frac12 } \lb \xi_{0} \rb^{ a }.
\end{equation*}
The argument closes as above as long as $a < \frac12(6b' - \frac52)$. 

This completes the proof.

\subsection{Proof of Proposition \ref{duhamelEst}} \label{duhamelEstProof}

 Since $|\hat{N}|$ is unaffected by spatial translations, it suffices to consider the $H^\frac{s+1}3_t$ norm at $x=0$. Suppose first that $s \leq 2-3b'$.  We have 
 \begin{align*}
\int_0^t W^{t-t'}_\R[N(x,t')] \d t' \Big|_{x=0} 
&= 
\int_{-\infty}^\infty \int_0^t e^{i(t-t')\xi^3}N(\hat{\xi},t') \d t' \d \xi \\
&= 
\int_{-\infty}^\infty \int_{-\infty}^\infty \int_0^t e^{i\tau t' + i(t-t')\xi^3}\hat{N}(\xi,\tau) \d t' \d \tau \d \xi \\
&=
- i \int_{-\infty}^\infty \int_{-\infty}^\infty \frac{e^{i\tau t} - e^{it \xi^3}}{\tau - \xi^3} \hat{N}(\xi,\tau)  \d \tau \d \xi.
 \end{align*}
We consider the regions where $|\tau - \xi^3| \lesssim 1$ and where $|\tau - \xi^3| \gtrsim 1$ separately.  When $|\tau - \xi^3| \lesssim 1$, we Taylor expand $e^{i\tau t} - e^{it \xi^3}$ and argue just as in the proof of \cite[Prop. 3.4]{ET} to obtain the desired $\| N\|_{X^{s,-b'}}$ bound. It remains to bound
\[ \left\| \iint_{|\tau-\xi^3| >1} \frac{e^{i\tau t} - e^{it \xi^3}}{\tau - \xi^3} \hat{N}(\xi,\tau)  \d \tau \d \xi\right\|_{H^\frac{s+1}3_t}. \]
We estimate the integral involving $e^{it\tau}$ and that involving $e^{it \xi^3}$ separately. First, we see 
\begin{align*}
\Big\| \iint_{|\tau-\xi^3| >1} \frac{e^{i\tau t}}{\tau - \xi^3} \hat{N}(\xi,\tau)  \d \tau &\d \xi\Big\|_{H^\frac{s+1}3_t}
 \lesssim
 \left\| \lb \tau \rb^{\frac{s+1}{3}} \int \frac{1}{\lb \tau - \xi^3 \rb} \hat{N}(\xi,\tau)  \d \xi\right\|_{L^2_\tau} \\
 &\lesssim 
 \| N \|_{X^{s,-b'}} \; \sup_\tau \lb \tau \rb^{\frac{s+1}3} \| \lb \tau - \xi^3 \rb^{b'-1} \lb \xi \rb^{-s} \|_{L^2_\xi} \\
 &\lesssim
 \| N \|_{X^{s,-b'}}. 
\end{align*}
To obtain the final inequality, we considered $|\xi| \leq 1$ and $|\xi| \geq 1$ separately. When $|\xi | \leq 1$, we see that $\sup_\tau \lb \tau \rb^{\frac{s+1}3} \| \lb \tau - \xi^3 \rb^{2b'-2} \lb \xi \rb^{-2s} \|_{L^2_\xi} \lesssim \sup_\tau \lb \tau \rb^{\frac{s+1}3 + b' -1}$, which is finite if $s \leq 2 - 3b'$. When $|\xi| \geq1$, we note that 
\begin{align*}
\sup_\tau \; \lb \tau \rb^{\frac{s+1}3} \|& \lb \tau - \xi^3 \rb^{b'-1} \lb \xi \rb^{-s} \|_{L^2_{|\xi| \geq 1}}\\
&\lesssim
\sup_\tau \; \lb \tau \rb^{\frac{s+1}3} \left( \int_{|\xi| \geq 1} \lb \tau - \xi^3 \rb^{2b'-2} \lb \xi \rb^{-2s} \d \xi \right)^\frac12 \\
&\lesssim 
\sup_\tau \; \lb \tau \rb^{\frac{s+1}3} \left( \int_{|w| \geq 1} \lb \tau - w \rb^{2b'-2} \lb w \rb^{-\frac{2(s+1)}3} \d w \right)^\frac12 
\; \lesssim \; 1.
\end{align*}
The last inequality follows from Lemma \ref{calcEst} as long as $s \geq -1$. 

We now look at 
\[ \left\| \eta(t) \iint_{|\tau-\xi^3| >1} \frac{ e^{it \xi^3}}{\tau - \xi^3} \hat{N}(\xi,\tau)  \d \tau \d \xi\right\|_{H^\frac{s+1}3_t}. \]
On the region where $|\xi|\leq1$, we have 
\begin{align*}
  \Big\| \eta(t)& \iint_{|\xi| \leq 1} \frac{ e^{it \xi^3}}{\lb \tau - \xi^3 \rb } \hat{N}(\xi,\tau)  \d \tau \d \xi\Big\|_{H^\frac{s+1}3_t}\\
  &\lesssim
   \iint_{|\xi| \leq 1 } \frac{ \| \eta(t) e^{it \xi^3}\|_{H^\frac{s+1}3_t}}{\lb \tau - \xi^3 \rb } \left| \hat{N}(\xi,\tau)   \right| \d \tau \d \xi \\
   &\lesssim
   \| N \|_{X^{s,-b'}} \; \| \chi_{[-1,1]}(\xi) \lb \tau - \xi^3 \rb^{b'-1} \|_{L^2_{\xi,\tau}} 
   \lesssim  
   \| N \|_{X^{s,-b'}}.
\end{align*}
This estimate holds for any $b' < \frac12$. 
It remains to bound 
\begin{align*}
 \left\|  \iint_{|\xi|>1} \frac{ e^{it \xi^3}\hat{N}(\xi,\tau)}{\lb \tau - \xi^3 \rb}   \d \tau \d \xi\right\|_{H^\frac{s+1}3_t}
 &\lesssim
 \left\|  \iint \frac{ e^{it w}}{ \lb \tau - w \rb} \hat{N}(w^{1/3},\tau)  \frac{\d \tau \d w}{\lb w \rb^{2/3}} \right\|_{H^\frac{s+1}3_t} \\
 &\approx 
 \left\| \lb w \rb^{\frac{s-1}3}  \int \frac{\hat{N}(w^{1/3},\tau)}{ \lb \tau - w \rb}   \d \tau  \right\|_{L^2_w} \\
 &\lesssim
 \left\| \lb w \rb^{\frac{s-1}3}  \frac{\hat{N}(w^{1/3},\tau)}{ \lb \tau - w \rb^{b'}}    \right\|_{L^2_{w, \tau}} 
 \; \lesssim \; 
 \| N \|_{X^{s,-b'}}. 
\end{align*}
This completes the proof for $-1 \leq s \leq 2 - 3b'$. 

Reviewing the argument above, we see that it holds for large $s$ with the exception of the estimate on the term 
\begin{align} \label{eq:term1}
 \left\| \lb \tau \rb^{\frac{s+1}{3}} \int \frac{1}{\lb \tau - \xi^3 \rb} \hat{N}(\xi,\tau)  \d \xi\right\|_{L^2_\tau}. 
\end{align}
Let $R = \{ (\xi, \tau) \; : \; |\tau | \geq 1 \text { and } |\tau| \gg |\xi|^3 \}$. Then note that 
\[ \lb \tau \rb \lesssim \chi_R(\xi,\tau) \lb \tau - \xi^3 \rb + \lb \xi \rb^3. \]
Hence 
\begin{align*}
 \eqref{eq:term1} \lesssim \left\| \int \frac{\chi_R(\xi,\tau)}{\lb \tau - \xi^3 \rb^{\frac{2-s}{3}}} \hat{N}(\xi,\tau)  \d \xi\right\|_{L^2_\tau} + \left\| \int \frac{\lb \xi \rb^{s+1}}{\lb \tau - \xi^3 \rb} \hat{N}(\xi,\tau)  \d \xi\right\|_{L^2_\tau} . 
\end{align*}
The last term can be bounded by 
\[ \sup_\tau \left\| \frac{\lb \xi \rb}{\lb \tau - \xi^3 \rb^{1-b'}} \right\|_{L^2_\xi} \| N \|_{X^{s,-b'}} \lesssim  \| N \|_{X^{s,-b'}}, \]
where the control on the supremum of the $L^2_\xi$ norm comes from a calculus calculation.

\subsection{Proof of Proposition \ref{correctionEst}}\label{correctionEstProof}

For the correction term, we wish to obtain, for $s>0$, an estimate of the form 
\[ \left\| \int \chi_R(\xi_0, \tau_0) \lb \tau_0 - \xi_0^3 \rb^{\frac{s+a-2}{3}} | \hat{N}(\xi_0,\tau_0)| \d \xi_0 \right\|_{L^2_{\tau_0}} \lesssim \prod_{j=1}^4 \| u_j\|_{X^{s,b}}, \]
where $N(x,t) = \partial_x \left( u_1 u_2 u_3 u_4 \right)$. 
Writing the estimate in its dual form and introducing functions $f_j = \lb \xi \rb^s \lb \tau - \xi^3 \rb^{b} \hat{u_j}$, it amounts to showing
\begin{multline*} 
\left|    \iiiint\limits_{\sum \tau_j = 0}  \iiiint\limits_{\sum \xi_j = 0} \chi_R(\xi_0,\tau_0) h(\tau_0)\frac{|\xi_0| \prod_{j=1}^4 \lb \xi_j \rb^{-s}}{\lb \tau_0 - \xi_0^3 \rb^{\frac{2-(s+a)}3}}   \prod_{j=1}^4 \frac{f_j(\xi_j, \tau_j)}{ \lb \tau_j - \xi_j^3 \rb^{b} }
\d \xi_j \d \tau_j \right| \\
\lesssim  \| h \|_{L^2_\tau} \prod_{j=1}^4 \| f_j \|_{L^2_{\xi,\tau}}. 
\end{multline*}
Using the fact that we are constrained to the set $R$ and then the Cauchy-Schwartz inequality, the LHS of the above quantity is bounded by 
\begin{align*} 
\Bigg|  &  \iiiint\limits_{\sum \tau_j = 0}  \iiiint\limits_{\sum \xi_j = 0} \chi_R(\xi_0,\tau_0) h(\tau_0) \; \lb \tau_0  \rb^{\frac{s+a - 2}3} |\xi_0| \lb \xi_0 \rb^{-s} \prod_{j=1}^4 \frac{f_j(\xi_j, \tau_j)}{ \lb \tau_j - \xi_j^3 \rb^{b} }
\d \xi_j \d \tau_j \Bigg| \\
&\lesssim 
\| h\|_{L^2_\tau} \Biggl[ \sup_\tau \; \lb \tau \rb^{\frac{s+a -2}{3}} \left\| \chi_R(\xi,\tau) |\xi| \lb \xi \rb^{-s} \right\|_{L^2_\xi} \Biggr]\; \prod_{j=1}^4 \left\| \mathcal{F}^{-1} \left( \frac{f_j}{ \lb \tau - \xi^3 \rb^{b} } \right)  \right\|_{L^{8}_{x,t}} 
\end{align*}
By \eqref{eq:L8}, this is bounded by 
\[ \|h \|_{L^2_\tau}\prod_{j=1}^4 \| f_j\|_{L^2_{\xi,\tau}}\]
as long as 
\[ \sup_\tau \; \lb \tau \rb^{\frac{s+a -2}{3}} \left\| \chi_R(\xi,\tau) |\xi| \lb \xi \rb^{-s} \right\|_{L^2_\xi}  < \infty.\]
A calculation gives
\begin{align*}
\sup_\tau \; \lb \tau \rb^{\frac{s+a -2}{3}} \left\| \chi_R(\xi,\tau) |\xi| \lb \xi \rb^{-s} \right\|_{L^2_\xi} \lesssim
\sup_\tau \; \lb \tau \rb^{\frac{s+a -2}{3}} \lb \tau \rb^{\frac{\max\{3/2-s,0\}}3+},
\end{align*}
which is finite as long as $a < \frac12$ and $s+a < 2$.

\section*{Acknowledgements}

We thank Prof. Bingyu Zhang for useful discussions regarding this problem. {The first author was supported by NSF MSPRF \#1704865. The second author's work was supported by a grant from the Simons Foundation (\#355523 Nikolaos Tzirakis) and the Illinois Research Board RB18051.}

\end{document}